\documentclass[10pt,leqno]{amsart}
\usepackage{graphicx}
\usepackage{indentfirst,csquotes}

\usepackage{tikz}
\usetikzlibrary{arrows.meta, positioning}
\usepackage{microtype}
\usepackage{secdot}

\usepackage{amssymb,amsthm,amsmath,mathtools}
\usepackage{paralist,titlesec,fancyhdr,etoolbox}
\newtheorem{theorem}{Theorem}[section]
\newtheorem{definition}[theorem]{Definition}
\newtheorem{example}[theorem]{Example}
\newtheorem{lemma}[theorem]{Lemma}
\newtheorem{proposition}[theorem]{Proposition}
\newtheorem{remark}[theorem]{Remark}
\newtheorem{corollary}[theorem]{Corollary}

\newtheorem{theoremA}{Theorem}

\newtheorem*{theoremStar}{Theorem}

\usepackage{xcolor, hyperref}
\definecolor{gal}{RGB}{0, 7, 111}

\hypersetup{
    colorlinks=true,
    citecolor=gal,
    linkcolor=gal,
    urlcolor=gal
}

\makeatletter
\renewcommand\@biblabel[1]{\textcolor{gal}{[#1]}}
\makeatother

\titleformat{\section}
{\normalfont\Large\bfseries\centering}
{\thesection}{1em}{}

\newcommand{\C}{\mathcal{C}}
\newcommand{\Q}{\mathbb{Q}}
\newcommand{\A}{\mathcal{A}}
\newcommand{\Lie}{\mathcal{L}}
\newcommand{\Z}{\mathbb{Z}}

\usepackage{lipsum}
\usepackage{tikz-cd}

\begin{document}
\title{On commensurations of pro-$\mathcal{C}$ groups} 
\author[Initial Surname]{Pedro Cusinato}
\date{\today}
\email{cusinato.pedro@gmail.com}
\maketitle

\let\thefootnote\relax
\footnotetext{} 

\begin{abstract}
    We study commensurations of pro-$\C$ groups, where $\C$ is an extension-closed variety of finite groups. We characterize those commensurations of pro-$\C$ groups that can be extended to embeddings into a common overgroup, obtaining a pro-$\C$ analog of a recent theorem by Touikan. We also consider obtain a structural result on virtual retracts of pro-$\C$ groups.
\end{abstract} 

\bigskip

\section{Introduction}

Groups $G_1$ and $G_2$ are said to be commensurable if they share isomorphic finite-index subgroups. The notion of commensurability is central in geometric group theory, arising naturally in the study of hyperbolic manifolds \cite{mostow} \cite{prasad} \cite{hyperbolic}, lattices in semisimple Lie groups \cite{margulis} and $S$-arithmetic groups \cite{platonov}.

Recently, motivated by a question related to the tree-lattice proof \cite{treelatt} of Leighton's graph covering theorem \cite{leighton}, Touikan \cite{Minasyan} has considered the following problem for commensurable abstract groups: given a group $H$ and monomorphisms $\iota_i:H\to G_i$ with finite-index images for $i=1,2$, when does there exist an abstract group $K$ containing $G_1$ and $G_2$ as finite-index subgroups such that the diagram

\[
    \begin{tikzcd}[row sep=.8em, column sep=4em]
& G_1 \arrow[dr, "j_1"] & \\
H \arrow[ur, "\iota_1"] \arrow[dr, "\iota_2"] & & K \\
& G_2 \arrow[ur, "j_2"] &
    \end{tikzcd}
\]

\noindent commutes? If there is such $K$, we say that the commensuration $(\iota_1,\iota_2)$ admits a completion and we call $(j_1,j_2)$ a completion for the commensuration $(\iota_1,\iota_2)$. In this direction, Touikan obtains the following characterization for when a commensuration may be completed:

\begin{theoremStar}
    {$($\cite[A]{Minasyan}$)$}
    Let $(\iota_1,\iota_2)$ be a commensuration between groups $G_1$ and $G_2$ over a group $H$. A completion for $(\iota_1,\iota_2)$ exists if and only if
    \begin{itemize}
        \item[(i)] there is a finite-index subgroup $N$ of $H$ such that the images of $N$ in $G_1$ and $G_2$ are normal;
        \item[(ii)] the images of groups $G_1$ and $G_2$ in $\operatorname{Out}N$ together generate a finite group.
    \end{itemize}
\end{theoremStar}

Let $(\iota_1,\iota_2)$ be a commensuration between groups $G_1$ and $G_2$ over a group $H$. We may construct its corresponding free amalgamated product

\[
    G:=G_1*_H G_2,
\]
in which the images of $H$ in $G_1$ and $G_2$ under $\iota_1$ and $\iota_2$, respectively, are identified. 


\begin{remark}\label{AbstractResult}
        We can reformulate the previous theorem as follows: a commensuration $(\iota_1,\iota_2)$ of abstract groups $G_1$ and $G_2$ over a group $H$ with the corresponding amalgamated product $G$ can be completed if and only if
        \begin{itemize}
            \item[(i)] $[H:H_G]<\infty$;
            \item[(ii)] the image of $G$ in $\operatorname{Out}H_G$ is finite.  
        \end{itemize}
\end{remark}

In the context of profinite groups, the notion of commensurability is similar. Profinite groups are commensurable if they share isomorphic open subgroups. Commensurable profinite groups have the same subgroup growth, which is a topic of particular interest in the context of pro-$p$ and profinite groups \cite{growth}. We also mention that commensurability classes of compact $p$-adic analytic groups are in correspondence with isomorphism classes of finite-dimensional Lie algebras over the field of $p$-adic numbers \cite[II.V]{LALG}. In \cite{kamm}, the author investigates how much information the commensurability class of the profinite completion of an $S$-arithmetic group can provide. A systematic approach to the study of the commensurability class of a profinite group was developed in \cite{commensurators}.

In this paper, we consider the analogous completion problem for commensurations of profinite groups and, more generally (and more subtly), for commensurations of pro-$\C$ groups, where $\C$ is a class of finite groups closed under subgroups, quotients and extensions. That is, given commensurable pro-$\C$ groups $G_1$ and $G_2$ and continuous monomorphisms with open images
\[
    \iota_i: H\to G_i, \text{ for } i=1,2,
\]

\noindent for some pro-$\C$ group $H$, we seek a pair of continuous monomorphisms $(j_1,j_2)$ from $G_1$ and $G_2$, respectively, to a common pro-$\C$ group $K$ such that $j_i(G_i)$ is open in $K$ for $i=1,2$ and $j_1\circ\iota_1=j_2\circ\iota_2$. For a commensuration $(\iota_1,\iota_2)$ between pro-$\C$ groups $G_1$ and $G_2$ over a pro-$\C$ group $H$ we also may consider its corresponding free pro-$\C$ amalgamated product (see \cite[9.2]{RZ})
\[
    G:=G_1\amalg_H G_2.
\]
However, in contrast with the abstract free amalgamated product, free pro-$\C$ amalgamated products need not to contain faithful images of the groups involved in the construction, especially when $\C$ is not the class of all finite groups. Following \cite{RZ}, we say that $G$ is proper if the natural images of $G_1$ and $G_2$ in it are faithful.  


We prove a pro-$\C$ version of Theorem \cite[A]{Minasyan}:

 \begin{theoremA}\label{TeoA}
 For a class $\C$ of finite groups closed under subgroups, quotients and extensions, let $(\iota_1,\iota_2)$ be a commensuration between pro-$\mathcal{C}$ groups $G_1$ and $G_2$ over a pro-$\C$ group $H$ and let
 \[
 G=G_1\amalg_H G_2
 \]
 be the associated free pro-$\mathcal C$ amalgamated product. Denote by $H_G$ the normal core of $H$ in $G$. The commensuration $(\iota_1,\iota_2)$ admits a completion if and only if
 \begin{itemize}
     \item[(i)] $G$ is a proper amalgamated product of $G_1$ and $G_2$;
     \item[(ii)] the image of $G$ in $\operatorname{Out}H_G$ is finite.
 \end{itemize}
 \end{theoremA}


 
A closed subgroup of a profinite group is said to be a virtual retract if it is a retract of an open subgroup. We prove a criterion for completion of a commensuration of profinite groups in terms of virtual retracts:

\begin{theoremA}\label{TeoF}
    Let $(\iota_1,\iota_2)$ be a commensuration between profinite groups $G_1$ and $G_2$ over a profinite group $H$. Then $(\iota_1,\iota_2)$ admits a completion if and only if $H$ is a virtual retract of the free profinite amalgamated product $G:=G_1\amalg_H G_2$.
\end{theoremA}

As a consequence of the technical results needed for the proof of Theorem \ref{TeoA}, we also obtain the following structural result about normal virtual retracts on profinite groups. This result is a profinite version of one of the main theorems of \cite{Minasyan}:

\begin{theoremA}\label{TeoC}
     Let $G$ be a profinite group with a closed normal subgroup $N$ that is a virtual retract of $G$, and suppose that $G/N$ is finitely generated. Then for every open subgroup $H$ of $G$ containing $N$, there exists a closed normal finitely generated subgroup $M$ of $G$ such that
 \[
 M\leq H, \qquad M\cap N=1, \qquad MN \text{ is open in } G.
 \]
\end{theoremA}


\noindent \textbf{Outline.} In Section \ref{sec:technical} we establish the technical results needed for the proofs of Theorems \ref{TeoA} and \ref{TeoC}. In Section \ref{sec:commensurations} we explore commensurations of pro-$\C$ groups and prove Theorems \ref{TeoA}, \ref{TeoF} and \ref{TeoC}. In Section \ref{sec:examples} we collect examples of profinite groups with finite outer automorphism group so that condition $(ii)$ of Theorem \ref{TeoA} becomes automatic. In Section \ref{sec:commensurator} we give an application of the commensurator of a profinite group (as defined in \cite{commensurators}) to obtain a criterion for the existence of a completion for a commensuration between groups with trivial virtual center. In Section \ref{sec:graphs} we conclude the paper by generalizing Theorem \ref{TeoA} to commensurating graphs of groups, establishing pro-$\C$ versions of results presented in the Appendix of \cite{Minasyan}.

\section{Technical background}\label{sec:technical}

In this section we prove some technical results that are direct analogs of results about abstract groups in \cite{Minasyan}. The goal here is to establish the following:

\begin{theorem}\label{TechA}
    Let $G$ be a profinite group containing an open normal subgroup $H$ that can be written as an internal direct product $N\times K$, where $N$ is normal in $G$ and $K$ has finitely generated abelianization. Then there exists a closed subgroup $K'$ of $G$ with the following properties:
    \begin{itemize}
        \item[(i)] $G$ is virtually an internal direct product $N\times K'$;
        \item[(ii)] $K'$ is contained in $H$;
        \item[(iii)] $K'$ is normal in $G$.
    \end{itemize}
\end{theorem}

We start with a preliminary result on profinite modules.

\begin{lemma}
    Let $\Gamma$ be a finite group and $M$ a profinite $\Gamma$-module. If $M$ can be written as an internal direct sum $Z\oplus L$, where $Z$ is a $\Gamma$-submodule and $L$ is a finitely generated subgroup of $M$ that is not necessarily $\Gamma$-invariant, then there exists a finitely generated $\Gamma$-submodule $L'$ of $M$ such that $Z\oplus L'$ is open in $M$.
\end{lemma}

\begin{proof}
    We use additive notation. The short exact sequence
    \[
    0\to Z\to M\xrightarrow{\pi} M/Z\cong L\to 0
    \]
    splits via an isomorphism $\alpha:M/Z\to L\leq M$. However, $\alpha$ is not $\Gamma$-equivariant in general. Since $M/Z$ is a finitely generated profinite abelian group, we may write
    \[
    M/Z= \prod_p A_p,
    \]
    where $A_p$ is the unique Sylow pro-$p$ subgroup of $M/Z$. Let $n$ be the order of $\Gamma$. The subgroup
    \[
    T =\{x\in M/Z, nx=0\}
    \]
    is finite, and we may therefore take an open subgroup $V_0$ of $M/Z$ such that $V_0\cap T = 0$. Set
    \[
        V = \bigcap_{g\in\Gamma} g\cdot V_0.
    \]
    It is clear that $V$ is $\Gamma$-invariant and the finiteness of $\Gamma$ implies that $V$ is open in $M/Z$. Also, since $V\cap T=0$, multiplication by $n$ is injective on $V$. Let $U=nV$, which is also open in $M/Z$. We define a homomorphism $\beta:U\to M$ by
    \[
        \beta(u):=\sum_{g\in \Gamma} g\cdot\alpha(g^{-1}\cdot v),
    \]
    where $v$ is the unique element of $V$ such that $nv=u$. It is standard to check that $\beta$ is a continuous $\Gamma$-equivariant homomorphism from the $\Gamma$-module $U$ to $M$. Also, since $\pi\circ\beta = \operatorname{id}_U$, $\beta$ is injective. Set $L'=\beta(U)$ so that $\pi^{-1}(U)=Z+L'$. Since $U$ is open in $M/Z$, $Z+L'$ is open in $M$. Finally, if $x\in Z\cap L'$, $x=\beta(u)$ for some $u\in U$. Therefore
    \[
        0=\pi(x)=\pi\circ\beta(u)=u,
    \]
    which implies that $x=0$.
\end{proof}

\begin{proof}[Proof of Theorem \ref{TechA}]
    Let $Z$ be the center of $N$ so that $ZK=C_G(N)\cap H$. Since both $C_G(N)$ and $H$ are normal in $G$, the same holds for $ZK=Z\times K$. Therefore the derived subgroup $[ZK,ZK]=[K,K]$ of $ZK$ is normal in $G$, and we may consider the quotient
    \[
        M= \frac{ZK}{[ZK,ZK]}=Z\times\frac{K}{[K,K]}
    \]
    as a profinite $G$-module. Since both $H$ and $K$ act trivially on $M$, $M$ has a natural profinite $G/H$-module structure. We denote the finite group $G/H$ by $\Gamma$. Notice that, while $Z$ is indeed a $\Gamma$-submodule of $M$, in general $K/[K,K]$ is not $\Gamma$-invariant. But since $K/[K,K]$ is finitely generated, we may apply the previous lemma to obtain a $\Gamma$-submodule $L'$ of $M$ such that $Z\oplus L'$ is open in $M$.

    We again regard $M$ as a profinite $G$-module. Since $L'$ is normal in $G/[K,K]$, its preimage $K'$ is normal in $G$. Also, since $ZK'$ the preimage of $Z\oplus L'$ in $M$, it is open in $ZK$ and, since $Z\cap [K,K]=1$ in $G$ and $Z\cap L'=0$ in $M$, $Z\cap K'=1$ .

    It follows that $ZK'$ contains an open subgroup of $K$. Hence $NK'=NZK'$ has finite-index in $H$. Moreover, $K'\leq ZK$ centralizes $N$. Also, $N\cap ZK= Z(N\cap K)=Z$, so $N\cap K'=Z\cap K'=1$, which concludes the proof that $NK'=N\times K'$.
\end{proof}

We finish this section with another technical lemma which will be useful in Sections \ref{sec:commensurations} and \ref{sec:graphs}.

\begin{lemma}\label{LemmaOut}
    Let $N$ be a closed normal subgroup of a pro-$\C$ group $G$ such that the image of $G$ in $\operatorname{Out}N$ is finite and let $J$ be a finitely generated free pro-$\C$ subgroup of $G$ such that $J\cap N=1$. Then there exists a subgroup $K\leq NJ$ such that $NK=N\times K$ and the composition $K\cong NK/N \leq NJ/N\cong J$ has open image in $J$. In particular, $NK$ is open in $NJ$.
\end{lemma}

\begin{proof}
    Let $p:J\to\operatorname{Out}N$ be the natural map, and let $K'$ be an open subgroup of $J$ contained in the kernel of $p$. Then $NK'$ is an open subgroup of $NJ$.

    Since $K'$ is open in $J$, $K'$ is a finitely generated free pro-$\C$ group. Let $\{k_1',\dots,k_n'\}$ be a basis for $K'$. Since $K'$ lies inside $\ker p$, for each $i=1,\dots,n$, there exists $t_i\in N$ such that, for any $n\in N$,
    \[
        n^{k_i'}=n^{t_i}.
    \]
    For $i=1,\dots, n$, let $k_i=k_i' t_i^{-1}$ and let $K=\langle k_1,\dots,k_n\rangle$. We note that the product $NK$ equals $NK'$. But $K$ centralizes $N$, since for any $n\in N$,
    \[
        n^{k_i}=n^{k_i't_i^{-1}}=(n^{k_i'})^{t_i^{-1}}=n^{t_it_i^{-1}}=n.
    \]

    Since $K'$ is free on $\{k_1',\dots,k_n'\}$, the assignment $k_i'\mapsto k_i$ extends to a continuous homomorphism $\varphi:K'\to NJ$. Let $\pi:NJ\twoheadrightarrow J$ be the quotient map modulo $N$. Since $\pi(k_i)=k_i'$, we have $\pi\circ\varphi|_{K'}=\operatorname{id}_{K'}$. That means that $\pi|_K$ is an isomorphism between $K$ and $K'$. Since $\ker(\pi|_K)=K\cap N$, we conclude that $K\cap N=1$.
\end{proof}

\section{Commensurations of pro-$\C$ groups}\label{sec:commensurations}

We say that pro-$\C$ groups $G_1$ and $G_2$ are \textit{commensurable} if there exist a pro-$\C$ group $H$ and continuous monomorphisms
\[
    \iota_1:H \to G_1 \text{ and } \iota_2:H\to G_2
\]
\noindent with open images. A pair $(\iota_1,\iota_2)$ is called a \textit{commensuration} between the pro-$\C$ groups $G_1$ and $G_2$. We will say that $(\iota_1,\iota_2)$ is a commensuration \textit{between} $G_1$ and $G_2$ \textit{over} $H$. We say that the commensuration $(\iota_1,\iota_2)$ is normal if $\iota_i(H)$ is a normal subgroup of $G_i$ for $i=1,2$.

A \textit{co-commensuration} (in the sense of \cite{Minasyan}) of pro-$\C$ groups is a pair $(j_1,j_2)$ of continuous monomorphisms
\[
    j_1:G_1\to K \text{ and } j_2:G_2\to K
\]
with open images, where $K$ is a pro-$\C$ group. We say that pro-$\C$ groups $G_1$ and $G_2$ are \textit{co-commensurable} if there exists a co-commensuration $(j_1,j_2)$ whose domains are $G_1$ and $G_2$, respectively. 

It is clear that co-commensurable groups are commensurable. The converse, however, is not true in general.

We say a commensuration $(\iota_1,\iota_2)$ between pro-$\C$ groups $G_1$ and $G_2$ can be \textit{completed} if there exists a co-commensuration $(j_1,j_2)$, with domains $G_1$ and $G_2$, respectively, such that the diagram 

\[
    \begin{tikzcd}[row sep=.8em, column sep=4em]
& G_1 \arrow[dr, "j_1"] & \\
H \arrow[ur, "\iota_1"] \arrow[dr, "\iota_2"'] & & K \\
& G_2 \arrow[ur, "j_2"'] &
\end{tikzcd}
\]

\noindent commutes. We say such a co-commensuration is a \textit{completion} of the commensuration $(\iota_1,\iota_2)$.


\begin{remark}
    A profinite group is called strongly complete if all of its subgroups of finite-index are open. If $G$ is a profinite group and $H$ is an open subgroup of $G$, then $G$ is strongly complete if and only if so is $H$. For strongly complete profinite groups (with $\C$ the class of all finite groups), the notion of commensurability coincides with the abstract notion. A remarkable theorem of Nikolov and Segal states that all finitely generated profinite groups are strongly complete \cite{strong1} \cite{strong2}.
\end{remark}


We recall that for a free pro-$\C$ amalgamated product $G:=G_1\amalg_H G_2$, the natural images of $G_1$ and $G_2$ in $G$ need not to be faithful. We say that a the free pro-$\C$ amalgamated product of $G_1$ and $G_2$ over $H$ is \textit{proper} if $\varphi_i:G_i\to G_1\amalg_H G_2$ are injective for $i=1,2$.

\begin{example}\label{Ex1}
    Let $F_n$ and $\widehat{F_n}$ denote the free group and the free profinite group of rank $n$, respectively. In \cite{free336}, the author constructs an abstract commensuration between the groups $G_1\cong F_3$ and $G_2\cong F_3$ over $F_{13}$ such that the corresponding amalgamated product $F_3*_{F_{13}} F_3$ has no finite quotients. By passing to profinite completions, we obtain a commensuration
    \[
\begin{tikzcd}[row sep=.8em, column sep=4em]
& \widehat{F_3}  & \\
\widehat{F_{13}} \arrow[ur, "\iota_1"] \arrow[dr, "\iota_2"'] & &  \\
& \widehat{F_3} &
\end{tikzcd}
\]
    of profinite groups whose the associated amalgamated product $\widehat{F_3}\amalg_{\widehat{F_{13}}}\widehat{F_3}$ is trivial.
\end{example}

\begin{lemma}\label{PropernessofEmbed}
    If a commensuration $(\iota_1,\iota_2)$ between pro-$\C$ groups $G_1$ and $G_2$ over a pro-$\C$ group $H$ admits a completion, then the associated free pro-$\C$ amalgamated product $G:=G_1\amalg_H G_2$ is proper.
\end{lemma}

\begin{proof}
    Let $(j_1,j_2)$ be a completion of the commensuration $(\iota_1,\iota_2)$. By the universal property of the amalgamated product, we obtain the commutative diagram below:
\[
\begin{tikzpicture}[>=stealth]

\node (H) at (0,0) {$H$};
\node (G1) at (2,1) {$G_1$};
\node (G2) at (2,-1) {$G_2$};
\node (A)  at (4,0) {$G_1 \amalg_H G_2$};
\node (K)  at (7,0) {$K$};

\draw[->] (H) -- (G1);
\draw[->] (H) -- (G2);

\draw[->] (G1) -- node[above] {$\varphi_1$} (A);
\draw[->] (G2) -- node[above] {$\varphi_2$} (A);

\draw[->, dashed] (A) -- node[above] {$\Phi$} (K);

\draw[->, bend left=20] (G1) to (K);
\draw[->, bend right=20] (G2) to (K);

\end{tikzpicture}
\]
Since $j_1$ and $j_2$ are injective and $j_i=\Phi\circ\varphi_i$ for $i=1,2$, we conclude that $G_1\amalg_H G_2$ is proper.
\end{proof}

In particular, we conclude that the commensuration of free profinite groups in Example \ref{Ex1} cannot be completed.

\begin{lemma}\label{PropernessLem}
    Let $G:=G_1\amalg_H G_2$ be the amalgamated pro-$\C$ product of the commensurable pro-$\C$ groups $G_1$ and $G_2$ over a common open subgroup $H$. If $G$ is proper, then the normal core $H_G$ of $H$ in $G$ is open in $H$. In particular, $H_G$ is a common open normal subgroup of $G_1$ and $G_2$.
\end{lemma}

\begin{proof}
    Since $G$ is proper, there exist a set $\Lambda$ and, for each $i=1,2$, a collection $\mathcal{U}_i=\{U_{i,\lambda},\lambda\in\Lambda\}$ of open normal subgroups of $G_i$ such that
    \[
        \bigcap_{\lambda\in\Lambda} U_{i,\lambda}=1, \text{ for } i=1,2 \text{ and } U_{1,\lambda}\cap H=U_{2,\lambda}\cap H, \text{ for all } \lambda\in\Lambda
    \]
    (see \cite[9.2.4]{RZ}). Since $H$ is open in both $G_1$ and $G_2$ and $\bigcap_{\lambda}U_{i,\lambda}=1$, compactness implies that there is some finite subset $\Lambda_f$ of $\Lambda$ such that $W_i:=\bigcap_{\lambda\in\Lambda_f}U_{i,\lambda}\leq H$ for $i=1,2$. Since $W_1=W_1\cap H=W_2\cap H= W_2$, we conclude that $W_1=W_2$ is contained in $H_G$ and is open in both $G_1$ and $G_2$.
\end{proof}

We conclude from Lemmas \ref{PropernessofEmbed} and \ref{PropernessLem} that a commensuration of pro-$\C$ groups that can admits a completion must extend a normal commensuration in the following sense: if $(\iota_1,\iota_2)$ is a commensuration of pro-$\C$ groups $G_1$ and $G_2$ over a pro-$\C$ group $H$ and admits a completion, then there must exist an open subgroup $N$ of $H$ such that the induced commensuration $(\iota_1|_N,\iota_2|_N)$ has normal images in $G_1$ and $G_2$. 











For a commensuration $(\iota_1,\iota_2)$ between pro-$\C$ groups $G_1$ and $G_2$ over a pro-$\C$ group $H$, we obtain group homomorphisms $\mathcal{O}_i:G_i\to\operatorname{Out}H_G$ for $i=1,2$. The image of $G_i$ in $\operatorname{Out}H_G$ is finite since $H_G$ is open in $G_i$ for $i=1,2$ by Lemma \ref{PropernessLem}.

The group $\operatorname{Out}H_G$ is a Hausdorff, totally disconnected group with the topology induced by the compact-open topology on $\operatorname{Aut}H_G$. If $H_G$ is finitely generated, then $\operatorname{Aut}H_G$ is a profinite group. The reader may consult Section 4.4 of \cite{RZ} for details. If $G:=G_1\amalg_H G_2$ is proper, we may regard $H_G$ as a normal subgroup of $G$ and make $G$ act on $H_G$ by conjugation to obtain a continuous homomorphism $\mathcal{O}:G\to\operatorname{Out}H_G$. Notice that $\mathcal{O}(G)=\langle\mathcal{O}_1(G_1),\mathcal{O}_2(G_2)\rangle$.


\begin{lemma}\label{OutFin}
    Let $(\iota_1,\iota_2)$ be a commensuration of pro-$\C$ groups $G_1$ and $G_2$ over a pro-$\C$ group $H$. If $(\iota_1,\iota_2)$ admits a completion, then the image of $G:=G_1\amalg_H G_2$ in $\operatorname{Out}H_G$ is finite.
\end{lemma}

\begin{proof}
    Recall that $G$ is proper and $H_G$ is open in both $G_1$ and $G_2$ by Lemma \ref{PropernessLem}. Let $(j_1,j_2)$ be a completion of the commensuration $(\iota_1,\iota_2)$. We may assume $K=\operatorname{codom}(j_1)=\operatorname{codom}(j_2)$ is generated by the images of $G_1$ and $G_2$. The image of $H_G$ in $K$ is normal, and thus we obtain a natural homomorphism $\varphi:K\to\operatorname{Out}H_G$. Since $H_G$ is open in $K$, $\varphi(K)=\langle\mathcal{O}(G_1),\mathcal{O}(G_2)\rangle$ is finite. To conclude that the image of $G$ in $\operatorname{Out}H_G$ is finite, recall that $\mathcal{O}(G)=\langle\mathcal{O}_1(G_1),\mathcal{O}_2(G_2)\rangle$. 
\end{proof}

We show an example of a normal commensuration for which the above condition does not hold.

\begin{example}\label{ExGL}
    Let $\pi$ be a finite set of primes and let $H=\prod_{p\in\pi}\Z_p^{n_p}$, where $n_p\geq 0$ for every $p\in\pi$. Let $K_1$ and $K_2$ be distinct maximal finite subgroups of $\operatorname{Aut}H\cong\prod_{p\in\pi} \operatorname{GL}_{n_p}(\Z_p)$. Consider the profinite groups $G_i:=H\rtimes K_i$ with the natural action. Let $\iota_i:H\to G_i$ be the natural inclusions for $i=1,2$. Then $(\iota_1,\iota_2)$ is a normal commensuration between $G_1$ and $G_2$, but $(\iota_1,\iota_2)$ does not admit a completion. In fact, the subgroup generated by the images of $G_1$ and $G_2$ in $\operatorname{Out}H=\operatorname{Aut}H$ is
    \[
        K:=\langle K_1, K_2\rangle,
    \]  
    which is infinite since $K_1$ and $K_2$ are distinct maximal finite subgroups of $\operatorname{Aut}H$.
\end{example}

The following stronger statement holds:

\begin{example}
    Let $G_1$ and $G_2$ as in \ref{ExGL}. If $K_1$ and $K_2$ are not conjugate in $\operatorname{Aut}H$, then there is no normal commensuration between $G_1$ and $G_2$, over any profinite group $N$, for which $\mathcal{O}(G_1\amalg_N G_2)$ is finite. In particular, the groups $G_1$ and $G_2$ are not co-commensurable.
\end{example}

\begin{proof}
    Let $(\iota_1,\iota_2)$ be a normal commensuration between $G_1$ and $G_2$ over a profinite group $N$. We may find an open subgroup $M\leq_o N$ such that $\iota_i(M)$ is normal in $G_i$ and contained in $H$ for $i=1,2$. Also, notice that $M\cong H$. It is clear that $(\iota_1,\iota_2)$ admits a completion if and only if $(\iota_1|_M,\iota_2|_M)$ does. The action of $K_i$ on $M$ is faithful for $i=1,2$. Hence, the images of $G_1$ and $G_2$ generate the subgroup
    \[
        K:=\langle K_1^{\varphi_1},K_2^{\varphi_2}\rangle, \text{ for some $\varphi_1,\varphi_2\in \operatorname{Aut}H$},
    \]
    which is infinite since $K_1$ and $K_2$ are non-conjugate maximal finite subgroups of $\operatorname{Aut}H$.
\end{proof}

We can now give the proof of Theorem \ref{TeoA}.

\begin{proof}[Proof of Theorem \ref{TeoA}]
    By Lemmas \ref{PropernessofEmbed} and \ref{OutFin}, a commensuration that admits a completion satisfies $(i)$ and $(ii)$. Conversely, if a commensuration satisfies $(i)$ and $(ii)$, we denote by $G$ the associated free pro-$\C$ amalgamated product. Since $G$ is proper, we may view $H$ as a subgroup of $G$. By Lemma \ref{PropernessLem}, $H_G$ is open in $G_1$ and $G_2$. The quotient $G/H_G$ is naturally isomorphic to the free pro-$\C$ amalgamated product
    \[
        (G_1/H_G)\amalg_{H/H_G}(G_2/H_G)
    \]
    which is virtually a free pro-$\C$ group since $G_1/H_G$, $G_2/H_G$ and $H/H_G$ are finite groups in $\C$. Let $J$ be an open normal subgroup of $G/H_G$ which is free pro-$\C$. Since $G/H_G$ is finitely generated, so is $J$.

    Let $\pi$ be the canonical projection $G\twoheadrightarrow G/H_G$. Since $J$ is free, the short exact sequence
    \[
        1\to H_G\to\pi^{-1}(J)\to J\to 1
    \]
    splits; thus $\pi^{-1}(J)=H_G\rtimes J'$, in which $J'$ is a lift of $J$ to $\pi^{-1}(J)$. By $(ii)$, the image of $G$ in $\operatorname{Out}H_G$ is finite and we may therefore apply Lemma \ref{LemmaOut} to find a subgroup $K'$ of $H_GJ'$ such that $H_G K'=H_G\times K'$ and $\pi(K')$ is an open subgroup of $J$.

    Since $J$ is finitely generated, it has only finitely many open subgroups having the same index in $J$ as $\pi(K')$. We may therefore pick an open characteristic subgroup $\overline{K}$ of $J$ inside $\pi(K')$. We set $K:=\pi^{-1}(\overline{K})\cap K'$. Since $\overline{K}$ is characteristic in $J$, it is normal in $G/H_G$ and therefore $\pi^{-1}(\overline{K})=H_GK=H_G \times K$ is an open normal subgroup of $G$.

    We are in the setting of Theorem \ref{TechA}: $H_G$ is a normal subgroup of $G$, and $K$ is a finitely generated subgroup of $G$ such that $H_GK=H_G\times K$ is an open normal subgroup of $G$. Therefore we can pick a subgroup $K_\star$ of $H_GK$ such that $K_\star$ is normal in $G$, $H_G K_\star=H_G\times K_\star$ is open and $\pi(K_\star)$ is open in $J$.

    We will show that both $G_1$ and $G_2$ are naturally embedded as open subgroups of $G/K_\star$. Notice that $\pi(K_\star)\leq J$ and therefore $\pi(K_\star)$ is torsion-free. Since the images of $G_1$ and $G_2$ in $G/H_G$ are finite and $K_\star\cap H_G=1$, $K_\star$ intersects both $G_1$ and $G_2$ trivially. We may therefore identify $H$, $G_1$ and $G_2$ as open subgroups of $G/K_\star$. To finish the proof, it remains to show that $G_1$ and $G_2$ are open in $G/K_\star$. But we may see that $G/H_GK_\star$ is finite because $\pi(K_\star)$ is open in $J$, and $J$ is open in $G/H_G$. Hence, we see that $H_GK_\star\cong H_G$ (and consequently $H$, $G_1$ and $G_2$) is an open subgroup of $G/K_\star$.
\end{proof}

\begin{corollary}\label{precisodisso}
    Let $(\iota_1,\iota_2)$ be a commensuration between pro-$\C$ groups $G_1$ and $G_2$ over a pro-$\C$ group $H$. Assume that $(\iota_1,\iota_2)$ admits a completion. Then $G:=G_1\amalg_H G_2$ is virtually a direct product $H_G\times K_\star$ with $K_\star\cap H=1$.
\end{corollary}

By the above corollary, if a commensuration $(\iota_1,\iota_2)$ between pro-$\C$ groups $G_1$ and $G_2$ over a pro-$\C$ group $H$ admits a completion, then $H_G$ is a virtual direct factor of $G:=G_1\amalg_H G_2$. The converse also holds when $G$ is proper:

\begin{corollary}\label{commDIR}
    Let $(\iota_1,\iota_2)$ be a commensuration between pro-$\C$ groups $G_1$ and $G_2$ over a pro-$\C$ group $H$. Assume that the free pro-$\C$ amalgamated product $G:=G_1\amalg_H G_2$ is proper. Then $(\iota_1,\iota_2)$ admits a completion if and only if $H_G$ is a virtual direct factor of $G$.
\end{corollary}

\begin{proof}
    If $(\iota_1,\iota_2)$ admits a completion, then $H_G$ is a virtual direct factor of $G$ by Corollary \ref{precisodisso}. Conversely, if $G$ is virtually $H_G\times M$ for some closed subgroup $M$ of $G$, then since $M$ centralizes $H_G$, the natural image of $G$ in $\operatorname{Out}H_G$ factors through the finite group $G/(H_G\times M)_G$ and is therefore finite. So $(\iota_1,\iota_2)$ admits a completion by Theorem \ref{TeoA}.
\end{proof}

\begin{remark}
    We may prove Theorem \ref{TeoA} when $\C$ is the class of all finite groups without the technical results of Section \ref{sec:technical}. In fact, Lemmas \ref{PropernessofEmbed} and \ref{OutFin} show the necessity of conditions $(i)$ and $(ii)$. Conversely, if $(\iota_1,\iota_2)$ is a commensuration over $H$ satisfying $(i)$ and $(ii)$, then $(\iota_1|_{H_G},\iota_2|_{H_G})$ is a normal commensuration between $G_1$ and $G_2$ by Lemma \ref{PropernessLem}. We may see $(\iota_1|_{H_G},\iota_2|_{H_G})$ as an abstract commensuration between the  underlying abstract groups $G_1$ and $G_2$. By Theorem A of \cite{Minasyan}, this commensuration admits an abstract completion $(j_1,j_2)$. By taking the open subgroups of $H_G$ as a fundamental system of neighborhoods of $1$ in $K=\operatorname{codom}(j_1)=\operatorname{codom}(j_2)$, $(j_1,j_2)$ becomes the desired completion of the commensuration $(\iota_1,\iota_2)$ in the category of profinite groups. However, when $\C$ does not contain all finite groups, $K$ might not be pro-$\C$.
\end{remark}





If $(\iota_1,\iota_2)$ is a normal commensuration between profinite groups $G_1$ and $G_2$ over a finitely generated profinite group $H$, the free profinite amalgamated product $G_1\amalg_HG_2$ is proper. In such a case, the group $\operatorname{Out}H$ is profinite. Therefore:

\begin{corollary}\label{col000}
    A normal commensuration $(\iota_1,\iota_2)$ between finitely generated profinite groups $G_1$ and $G_2$ over a profinite group $H$ may be completed if and only if $\mathcal{O}:G_1\amalg_H G_2\to\operatorname{Out}H$ has finite image.
\end{corollary}

\begin{example}
    Let $\Gamma$ be an $S$-arithmetic group and denote by $\widehat{\Gamma}$ its profinite completion and by $\overline{\Gamma}$ its completion with respect to the congruence subgroup topology. If $\Gamma$ satisfies the congruence subgroup property, then the groups $\widehat{\Gamma}$ and $\overline{\Gamma}$ are co-commensurable.

    Indeed, let $p:\widehat{\Gamma}\to\overline{\Gamma}$ be natural epimorphism, and let $C$ be the kernel of $p$. If $C$ is finite, then there is an open normal subgroup $N$ of $\widehat{\Gamma}$ such that $N\cap C =1$. Notice that $N$ is naturally isomorphic to an open subgroup of $\overline{\Gamma}$ and that $[N,C]=1$, so that the following diagram commutes:
\[
\begin{tikzcd}
\widehat{\Gamma} \arrow[d, ""'] \arrow[dr, ""] & \\
\overline{\Gamma} \arrow[r, ""'] & \operatorname{Out} N
\end{tikzcd}
\]
Therefore the image of $\widehat{\Gamma}$ in $\operatorname{Out}N$ is equal to the image of $\overline{\Gamma}$ in $\operatorname{Out}N$, which is finite. Since $N$ is finitely generated and normal in both $\widehat{\Gamma}$ and $\overline{\Gamma}$, the result follows from Corollary \ref{col000}.

A well-known conjecture by Serre seeks to classify $S$-arithmetic groups satisfying the congruence subgroup property in terms of their $S$-rank \cite{serreA}. While the general statement is still open, many important cases have been proved \cite{rag} \cite{bamise}.

\end{example}

Since the images of $G_1$ and $G_2$ in $\operatorname{Out}H$ are always finite, condition $(ii)$ of Theorem \ref{TeoA} holds whenever $\operatorname{Out}H$ is nilpotent or locally finite. A landmark result of Zelmanov is that torsion profinite groups are locally finite (see \cite{zelmanov}). Therefore, condition $(ii)$ of Theorem \ref{TeoA} holds if the group $\operatorname{Out}H$ is either a torsion profinite group or a nilpotent group. In Section \ref{sec:examples} we collect examples of profinite groups $H$ with finite outer automorphism group.

We finish this section with a few examples of profinite commensurations that can be completed.

\begin{corollary}\label{col21}
    Let $G_1$ and $G_2$ be virtually procyclic pro-$p$ groups which are commensurable.  Then $G_1$ and $G_2$ are co-commensurable. Further, every commensuration between $G_1$ and $G_2$ over a procyclic subgroup admits a completion.
    
\end{corollary}

\begin{proof}
    Let $H$ be a common open subgroup of $G_1$ and $G_2$ which is procyclic and let $H_i$ be the normal core of $H$ in $G_i$. Since every subgroup of $H$ is characteristic, $K:=H_1\cap H_2$ is a common open normal subgroup of $G_1$ and $G_2$. Since $K$ is procyclic, $\operatorname{Out}K$ is abelian and therefore condition $(ii)$ of Theorem \ref{TeoA} holds. Also, since $H$ is procyclic, the free pro-$p$ amalgamated product $G_1\amalg_H G_2$ is proper by \cite[3.2]{ribes}.
\end{proof}

\begin{remark}
    As in Corollary \ref{col21}, a normal commensuration between profinite groups over a procyclic group can be completed.
\end{remark}

\begin{example}
    Let $H$ be a finitely generated pro-$p$ group and $Q_1,Q_2\leq\operatorname{Aut}(H)$ be finite groups of order prime to $p$. Let
    \[
        G_i:=H\rtimes Q_i, \text{ for } i=1,2.
    \]
    The natural homomorphism $\Delta:\operatorname{Aut}H\to\operatorname{Aut}(H/\Phi(H))\cong\operatorname{GL}_n(\mathbb{F}_p)$ induces a short exact sequence
    \[
        1\to N\to \operatorname{Out}H\to\Delta(\operatorname{Aut}(H))=\Delta\to1
    \]
    in which the kernel $N$ is pro-$p$ (see \cite[4.5.5]{RZ} or \cite[5.5]{DDMS}). Suppose that the subgroup generated by the images of $Q_1$ and $Q_2$ in $\Delta\leq\operatorname{GL}_n(\mathbb{F}_p)$ has order prime to $p$. Then the profinite groups $G_1$ and $G_2$ are co-commensurable.

    Indeed, let $\pi:\operatorname{Out}H\to\Delta$ be the natural epimorphism and $Q$ be the subgroup generated by the images of $Q_1$ and $Q_2$ in $\Delta$. The preimage of $Q$ in $\operatorname{Out}H$ yields the short exact sequence
    \[
        1\to N\to \pi^{-1}(Q)\to Q\to 1,
    \]
    which splits by the profinite version of Schur-Zassenhaus theorem (see \cite[2.3.15]{RZ}). Furthermore, all complements of $N$ in $\pi^{-1}(Q)$ are conjugate. Denote by $\overline{Q_i}$ the image of $Q_i$ in $\operatorname{Out}H$ for $i=1,2$. The subgroups $\overline{Q_1}$ and $\overline{Q_2}$ lie in complements $C_1$ and $C_2$ of $N$ in $\pi^{-1}(Q)$. Let $\psi\in N$ be such that $C_2^\psi=C_1$, and let $\varphi$ be a representative of $\psi$ in $\operatorname{Aut}H$. Set $G_2':=H\rtimes Q_2^\varphi$ so that the group generated by the images of $G_1$ and $G_2'$ in $\operatorname{Out}H$ is contained in $C_1\cong Q$, which is finite. By Theorem \ref{TeoA}, the (normal) commensuration between $G_1$ and $G_2'$ over $H$ admits a completion. The conclusion follows from the fact that $G_2\cong G_2'$. 
\end{example}

\subsection{Virtual retracts of pro-$\C$ groups.}\label{subsec:virtual_retracts}

Let $G$ be a profinite group. A closed subgroup $N$ of $G$ is said to be a virtual retract of $G$ if there is an open subgroup $K$ of $G$ containing $N$ and an epimorphism $\rho: K\to N$ which restricts to the identity on $N$.

If $N$ is a virtual retract of $G$, the group $G$ is virtually a semidirect product $S\rtimes N$, where $S=\ker\rho$. If $N$ is normal, then $G$ is virtually a direct product $S\times N$. While $S$ is indeed normal in $K$, it is not clear that $G$ is virtually a direct product $M\times N$, with $M$ normal in $G$.

Below we prove Theorem \ref{TeoC}, which states that if $G/N$ is finitely generated, we may conclude that $G$ is virtually a direct product of $N$ with another normal subgroup $M$ of $G$. Moreover, $M$ can be chosen inside any prescribed open subgroup of $G$ containing $N$.

\begin{proof}[Proof of Theorem \ref{TeoC}]
    Since $N$ is a virtual retract of $G$, there is an open subgroup $L\leq_o G$  that retracts onto $N$. Since $L\cap H$ is open in $G$, the subgroup $L'=\bigcap_{g\in G} (L\cap H)^g$ is an open normal subgroup of $G$, which contains $N$ and is contained in $L\cap H$. It follows that $L'$ also retracts onto $N$. Hence, there is a normal subgroup $K$ of $L'$ such that $K\cap N=1$ and $L'=KN$. Since $K$ and $N$ are normal in $L'$, $L'$ can be viewed as a direct product $K\times N$.

    It is clear that $K\cong KN/N$ embeds as an open subgroup of $G/N$ and is, therefore, finitely generated. So we may apply Theorem \ref{TechA} to obtain a finitely generated subgroup $M$ of $KN$ which is normal in $G$ and such that $MN=M\times N$ and $MN$ is open in $G$. Since $M\leq L'\leq L\cap H$, $M$ is contained in $H$.
\end{proof}

Our goal for the rest of the section is to prove Theorem \ref{TeoF}; we start with two preliminary results.

\begin{lemma}\label{propernessvirt}
    Let $G:=G_1\amalg_H G_2$ be a free profinite amalgamated product. If the natural map from $H$ to $G$ is injective, then $G$ is proper.
\end{lemma}

\begin{proof}
    We shall prove the existence of a set $\Lambda$ and, for each $i=1,2$, a collection $\mathcal{W}_i=\{W_{i,\lambda},\lambda\in\Lambda\}$ of open normal subgroups of $G_i$ such that
    \[
        \bigcap_{\lambda\in\Lambda} W_{i,\lambda}=1, \text{ for } i=1,2 \text{ and } W_{1,\lambda}\cap H=W_{2,\lambda}\cap H, \text{ for all } \lambda\in\Lambda,
    \]
    which suffices by \cite[9.2.4]{RZ}. Let $\mathcal{U}_{i}=\{N_\lambda, \lambda\in\Lambda_i\}$ be the set of all open normal subgroups of $G_i$ and let
    \[
        \Lambda:=\Lambda_1\times\Lambda_2.
    \]
    For each $\lambda=(\lambda_1,\lambda_2)\in\Lambda$, let $H_\lambda=N_{\lambda_1}\cap N_{\lambda_2}\cap H\trianglelefteq_o H$. Then, since $H$ is a closed subgroup of $G$, there exists some $U_\lambda\trianglelefteq_o G$ such that
    \[
        U_\lambda\cap H\leq H_\lambda\leq N_{\lambda_{i}}, \text{ for } i=1,2.
    \]
    Let $U_\lambda^i$ be the preimage of $U_\lambda$ in $G_i$, and set
    \[
        W_{i,\lambda}:= U_\lambda^i\cap N_{\lambda_i}.
    \]
    Since $H$ intersects the kernel of the natural maps $\varphi:G_i\to G$ trivially, we obtain that $W_{i,\lambda}\cap H= U_\lambda^i\cap N_{\lambda_i}\cap H=U_\lambda\cap H\cap N_{\lambda_i}=U_\lambda\cap H$. In particular,
    \[
        W_{1,\lambda}\cap H=W_{2,\lambda}\cap H.
    \]
    It is clear that $\bigcap_{\lambda\in\Lambda} W_{i,\lambda}=1$ for $i=1,2$.
\end{proof}

We can now prove Theorem \ref{TeoF}.

\begin{proof}[Proof of Theorem \ref{TeoF}]
    If $H$ is a virtual retract of $G$, then $G$ is proper by Lemma \ref{propernessvirt} and $G$ has an open subgroup of the form $M\rtimes H$. Then $M\times H_G=M\rtimes H_G$ is open in $G$, so $H_G$ is a virtual direct factor of $G$. Then the commensuration $(\iota_1,\iota_2)$ may be completed by Corollary \ref{commDIR}. Conversely, if $(\iota_1,\iota_2)$ admits a completion, then $G$ is virtually a direct product $H_G\times K_\star$ with $K_\star\cap H=1$ by Corollary \ref{precisodisso}. Then $K_\star H=K_\star\rtimes H$ is open in $G$, and $H$ is therefore a virtual retract of $G$.
\end{proof}

We also obtain the following pro-$\C$ version of Theorem \ref{TeoF}:

\begin{theorem}
    Let $(\iota_1,\iota_2)$ be a commensuration between pro-$\C$ groups $G_1$ and $G_2$ over a pro-$\C$ group $H$. Then $(\iota_1,\iota_2)$ admits a completion if and only if
    \begin{itemize}
        \item[(i)] the free pro-$\C$ amalgamated product $G:=G_1\amalg_H G_2$ is proper;
        \item[(ii)] $H$ is a virtual retract of $G$.
    \end{itemize}
\end{theorem}

\begin{proof}
    We can use the same argument as in the proof of Theorem \ref{TeoF}. Properness is not automatic, though.
\end{proof}

\section{Some profinite groups with small outer automorphism group}\label{sec:examples}

As established in Theorem \ref{TeoA}, a commensuration $(\iota_1,\iota_2)$ of pro-$\C$ groups $G_1$ and $G_2$ over a pro-$\C$ group $H$ admits a completion if and only if $(i)$ $G:=G_1\amalg_H G_2$ is proper and $(ii)$ $G$ has finite image in $\operatorname{Out}H_G$. The first condition can be difficult to verify, especially in the pro-$\C$ situation. Condition $(ii)$ holds automatically if $\operatorname{Out}H_G$ is a finite group. In the abstract case, many classes of infinite groups with finite outer automorphism group are known. We mention the fundamental groups of hyperbolic $n$-manifolds of finite volume for $n>2$ (this is a consequence of the remarkable Mostow rigidity \cite{mostow} \cite{prasad}), word-hyperbolic groups with connected and cutpair-free Gromov boundaries (see \cite[1.4]{levitt} and \cite{bowbow}), many mapping class groups of surfaces \cite{teich} and symmetric multi-GGS groups in the sense of \cite{ggs}. We also mention that the Grigorchuk 2-group \cite{grig} and the Gupta-Sidki $p$-group for $p>2$ \cite{said} \cite{ggs} have locally finite outer automorphism group, so that condition $(ii)$ of Theorem \ref{TeoA} is automatic. However, in the profinite setting, this condition appears to hold less frequently. For this reason we collect in this section some examples of profinite groups with finite outer automorphism group. A commensuration over any of the groups mentioned in this section admits a completion if and only if the associated free pro-$\C$ amalgamated product is proper.

We recall that a group is called \textit{complete} if its center is trivial and it has no outer automorphisms. Although this terminology is unfortunate in the context of profinite groups, we will adopt it in this section with the natural adaptation: a profinite group is complete if its center is trivial and it has no continuous outer automorphisms.

\subsection{Groups of automorphisms of spherically homogeneous trees.}\label{subsec:spherically_homogeneous_trees}

Let $\ell=(n_1,n_2,\dots)$ be a sequence of positive integers greater than or equal to 2. We shall construct a rooted tree $T_\ell$ associated with $\ell$. For each $k$, let $X_k$ be a set of cardinality $n_k$ and $L_k=\prod_{j\leq k} X_j$ (we take $L_0$ to be an one-element set). The vertex set of $T_\ell$ is the disjoint union of the sets $L_k$ and two vertices $v$ and $w$ of $T_\ell$ are connected (by a unique edge) if and only if either
\begin{itemize}
    \item[(i)] $v=(x_1,\dots,x_k)\in L_k$ and $w=(x_1,\dots,x_k,x_{k+1})\in L_{k+1}$ for some $k\geq1$;
    \item[(ii)] $v=(x_1,\dots,x_{k-1})\in L_{k-1}$ and $w=(x_1,\dots,x_{k})\in L_k$ for some $k\geq 2$; 
    \item[(iii)] $v\in L_0$ and $w\in L_1$; or
    \item[(iv)]  $v\in L_1$ and $w\in L_0$.
\end{itemize}

Let $\A_\ell$ be the group of all graph automorphisms of $T_\ell$. For each positive integer $n$, we denote by $\operatorname{Stab}(n)$ the intersection of $\operatorname{Stab}_{\A_\ell}(v)$ for all $v\in L_n$. We define a topology on $\A_\ell$ by taking $\operatorname{Stab}(n)$ as a fundamental system of neighborhoods of $1$. With this topology, $\A_\ell$ is a profinite group which may be expressed as an inverse limit
\[
    \varprojlim_k \A_\ell/\operatorname{Stab}(k)=\varprojlim_k S_{n_1}\wr S_{n_2}\wr\dots\wr S_{n_k}
\]

\noindent of iterated permutational wreath products of symmetric groups.

We define a natural partial order $\precsim$ on the set of vertices of $T_\ell$ by saying $v\precsim w$ if $v=(x_1,\dots,x_n)$ and $w=(x_1,\dots,x_n,x_{n+1},\dots,x_m)$ for some $m\geq n$. We define $T_\ell^v$ to be the unique subtree of $T_v$ containing all vertices $w$ such that $v\precsim w$.

Let $G$ be an abstract subgroup of $\A_\ell$. For each vertex $v\in T_\ell$, define the \textit{rigid stabilizer of} $v$ to be the subgroup

\[
    \operatorname{Rist}_G(v):=\bigcap_{w\notin T_\ell^v} \operatorname{Stab}_G(w).
\]

\noindent For each $n$, define the \textit{rigid stabilizer of level} $n$ to be the subgroup of $G$ generated by the subgroups $\operatorname{Rist}_G(v)$ for $v\in L_n$.

The group $G$ is said to be \textit{weakly branch} if for all $n$, $G$ acts transitively on $L_n$ and $\operatorname{Rist}_G(n)\neq 1$. The (topological) closure of $G$ in $\A_\ell$ is a weakly branch profinite group.

A subgroup $H$ of $\A_\ell$ is \textit{saturated} (see \cite{branch}) if for every $n$ there is a characteristic subgroup $H_n$ of $H$ contained in $\operatorname{Stab}_H(n)=\bigcap_{v\in L_n} \operatorname{Stab}_H(v)$ such that $H_n$ acts transitively on every subtree of the level $n$. The motivation for this definition is the following:

\begin{theorem}[\cite{branch}, 7.5]
    If $G\leq\A_\ell$ is a saturated weakly branch group, then $\operatorname{Aut}(G)$ coincides with the normalizer of $G$ in $\A_\ell$.
\end{theorem}

We obtain the following corollary:

\begin{corollary}
    If $G$ is a saturated weakly branch group that is closed and self-normalizing in $\A_\ell$, then $G$ is a complete profinite group. In particular, the group $\A_\ell$ is a profinite group with $\operatorname{Out}\A_\ell=1$.
\end{corollary}

\subsection{Infinite products of finite groups.}\label{subsec:infinite_products} A group $G$ is indecomposable if $G$ is not an internal direct product of proper subgroups. Given a family $\mathcal{G}$ of pairwise non-isomorphic finite, complete, indecomposable groups, the profinite group
\[
    \operatorname{Prod}(\mathcal{G}) = \prod_{G\in\mathcal{G}} G
\]
\noindent is complete and infinite if $\mathcal{G}$ is infinite.

Examples of indecomposable finite complete groups are the symmetric groups $S_n$ for $n\geq 3, n\neq 6$, the holomorph of $C_p$ for $p>2$, and infinitely many non-abelian finite simple groups. More generally, $\operatorname{Aut}(S)$ is complete for any non-abelian finite simple group $S$.

We recall Wielandt's theorem on the automorphism tower (see \cite{wielandt}), which is useful for constructing finite complete groups. Given a finite group $G$, we may construct the following sequence of finite groups
\[
    A_1:= G \text{ and } A_{n+1}=\operatorname{Aut}(A_n), \text{ for } n\geq1.
\]

\noindent Wielandt's theorem states that, provided $G$ is centerless, there exists $N\geq 1$ such that $A_N$ is complete.

Some products of finite groups are examples of profinite groups whose outer automorphism group are infinite, compact and torsion.

\begin{example}
    Let $\mathcal{S}$ be an infinite family of pairwise non-isomorphic non-abelian simple groups with outer automorphism groups of bounded order. The unrestricted product
    \[
        G=\prod_{S\in\mathcal{S}} S
    \]
    \noindent is a profinite group whose outer automorphism group is a compact and torsion group. If $\operatorname{Out}S\neq 1$ for infinitely many $S\in\mathcal{S}$, then $\operatorname{Out}G$ is infinite.
\end{example}

\begin{example}
    Let $A_n$ denote the alternating group on $n$ symbols. The group
    \[
        G = \prod_{n\geq5} A_n
    \]
    is a $2$-generated profinite group and
    \[
    \operatorname{Out}G \cong \prod_{\mathbb{N}} C_2
    \]
    \noindent is an infinite elementary abelian $2$-group of countable rank.
\end{example}

\subsection{Compact $p$-adic analytic groups.}\label{subsec:p-adic_analytic} Commensurability plays an important role in the study of compact $p$-adic analytic groups, which are profinite and virtually pro-$p$. A standard fact is that if $G_1$ and $G_2$ are compact $p$-adic analytic, then
\[
    G_1 \text{ and } G_2 \text{ are commensurable } \iff \mathcal{L}(G_1)\cong\mathcal{L}(G_2),
\]

\noindent where $\mathcal{L}(G)$ is the $\Q_p$-Lie algebra of a $p$-adic analytic group $G$.

A $p$-adic analytic group which is pro-$p$ is usually called an \textit{analytic pro-$p$} group. A finitely generated, torsion-free, and powerful pro-$p$ group is said to be \textit{uniform} (or \textit{uniformly powerful}). Uniform groups are analytic pro-$p$ groups. A stronger fact is that a topological group is $p$-adic analytic if and only if it contains an open uniform group.

For a compact $p$-adic analytic group $G$, $\operatorname{Aut}(G)$ is also compact and $p$-adic analytic. However, it can be much larger than $G$, as in the standard example
\[
    \operatorname{Aut}(\Z_p^n)=\operatorname{GL}_n(\Z_p),
\]
\noindent in which $\Z_p^n$ has dimension $n$ and its automorphism group has dimension $n^2$. However, for semisimple $p$-adic analytic groups, there are few automorphisms. All mentioned results are standard, and the proofs may be found in \cite{DDMS}.

A Lie algebra $\mathfrak{g}$ is \textit{complete} if all its derivations are inner and its center is zero. Semisimple Lie algebras over $\Q_p$ are complete, but there are complete Lie algebras that are not semisimple, such as the unique non-abelian Lie algebra of dimension 2.

\begin{proposition}
    If $G$ is a uniform pro-$p$ group with a complete Lie algebra, then $\operatorname{Out}(G)$ is finite.
\end{proposition}

\begin{proof}
    Since $G$ is a compact $p$-adic analytic group, so is $\operatorname{Aut}(G)$. The group $G$ embeds naturally as a closed subgroup of $\operatorname{Aut}(G)$, and we may identify $\operatorname{Out}(G)$ with $\operatorname{Aut}(G)/G$. Therefore, to show $\operatorname{Out}(G)$ is finite, it is sufficient to show that the dimension of $\operatorname{Aut}(G)$ as a $p$-adic analytic group is the same as the dimension of $G$. But since $G$ is uniform, $\operatorname{dim}\operatorname{Aut}(G)=\dim_{\Q_p}\operatorname{Der}(\Lie(G))$. Since $\operatorname{Der}(\Lie(G))$ consists just of inner derivations, $\operatorname{dim}\operatorname{Aut}(G)=\operatorname{dim}G$.
\end{proof}





\subsection{Profinite completion of some mapping class groups.}\label{subsec:mapping_class} Let $\Sigma_{g,n}$ be a closed orientable surface of genus $g$ with $n$ punctures. The extended mapping class group $\operatorname{Mod}^\pm(\Sigma_{g,n})$ is the group of isotopy classes of diffeomorphisms of $\Sigma_{g,n}$. These groups are residually finite, and their profinite completions play an important role in Grothendieck-Teichmüller theory \cite{hatcher}.

The groups $\operatorname{Mod}^\pm(\Sigma_{g,n})$ are, apart from a few exceptions, complete. In \cite{Boggi}, the authors proved that the profinite completion of $\operatorname{Mod}^\pm(\Sigma_{g,n})$ is complete when $g=0$ and $n\geq 5$:

\begin{theorem}
    Let $\Gamma:=\operatorname{Mod}^\pm(\Sigma_{0,n})$ be the extended mapping class group of the sphere with $n$ punctures. If $n\geq 5$, its profinite completion $\widehat{\Gamma}$ is a complete profinite group.
\end{theorem}

To our knowledge, this is the first example of a infinite, finitely generated, residually finite, complete group with complete profinite completion.

In the same paper, the authors produce another class of examples of profinite groups with finite outer automorphism groups. Let $\mathcal{M}_{0,n}$ be the moduli space of $n$-pointed curves of genus $0$ and let $(\mathcal{M}_{0,n})_\mathbb{R}$ denote $\mathcal{M}_{0,n}\times\operatorname{Spec}(\mathbb{R})$.

\begin{theorem}
    Let $n\geq 5$. The étale fundamental group $\pi_1^{\text{et}}((\mathcal{M}_{0,n})_\mathbb{R})$ has outer automorphism group isomorphic to the symmetric group $S_n$.
\end{theorem}

\subsection{The Nottingham group.}\label{subsec:nottingham} Let $\mathbb{F}_q$ be the finite field with $q$ elements. The Nottingham group $\mathcal{N}(q)$ is defined as the group of formal power series $f=\sum_{n\geq 0}a_n t^n\in\mathbb{F}_q[[t]]$ with $a_0=0$ and $a_1=1$, where the group law is composition by formal substitution. The group $\mathcal{N}(q)$ is a finitely generated pro-$p$ group ($q=p^n$), contains all countably-based pro-$p$ groups and is hereditarily just infinite unless $q= 2^n$ and $n>1$ \cite{camina}. The automorphism group of $\mathcal{N}(q)$ was fully described for $p\geq 5$ in \cite{nottingham}. It was shown that all automorphisms of $\mathcal{N}(q)$ are compositions of inner automorphisms and automorphisms induced by a field automorphism. Theorem 1.1 of \cite{nottingham} implies the following:

\begin{proposition}
    If $p\geq 5$, the group $\operatorname{Out}(\mathcal{N}(q))$ is finite of order $n(q-1)$.
\end{proposition}

\section{The commensurator}\label{sec:commensurator}

Let $G$ be a profinite group. A continuous isomorphism between two open subgroups of $G$ is a \textit{virtual automorphism} of $G$. We say that two virtual automorphisms are equivalent if they coincide on some open subgroup of $G$. The set of equivalence classes of virtual automorphisms of $G$ is naturally a group, called the \textit{commensurator} of $G$.

The virtual center $\operatorname{VZ}(H)$ of a profinite group $H$ is the subgroup consisting of elements whose centralizers are open. For an open subgroup $U$ of $G$, conjugation induces a canonical homomorphism
\[
    \iota_U: U\to\operatorname{Comm}(G)
\]

\noindent whose kernel is $\operatorname{VZ}(U)$. Let $\mathcal{U}$ be the family of open subgroups of $G$. We put a topology on $\operatorname{Comm}(G)$ by declaring $\{\operatorname{Im}(\iota_U),U\in\mathcal{U}\}$ to be a basis of open neighborhoods of $1$ — this is called the \textit{strong topology} of $\operatorname{Comm}(G)$ in \cite{commensurators}. We now collect some results about the commensurator.

\begin{proposition}[\cite{commensurators}]\label{commensurator}
    Let $G$ be a profinite group with $\operatorname{VZ}(G)=1$. Then $\operatorname{Comm}(G)$ is a locally compact, totally disconnected topological group (l.c.t.d.) in which $G$ embeds as an open subgroup. Further, it satisfies the following universal property: for every continuous monomorphism $\eta:G\to L$ into some l.c.t.d. group $L$ such that $\eta$ is a topological isomorphism from $G$ onto an open subgroup of $L$, there is a unique homomorphism $\eta_L:L\to\operatorname{Comm}(G)$ such that the diagram
    \[
    \begin{tikzcd}[row sep=3em, column sep=4em]
G \arrow[r, "\eta"] \arrow[dr, "\iota_G"] & L \arrow[d, dashed, "\eta_L"'] \\
& \mathrm{Comm}(G)
\end{tikzcd}
    \]
    commutes and $\operatorname{ker}(\eta_L)=\operatorname{VZ}(L)$.
\end{proposition}

\begin{remark}
    If $G$ has nontrivial virtual center, $\operatorname{Comm}(G)$ still has a property analogous to that in Proposition \ref{commensurator}. The difference is that $\iota_G$ is not injective and $\eta_L$ is not unique. Also, $\operatorname{Comm}(G)$ is Hausdorff if and only if $\operatorname{VZ}(G)$ is closed. See \cite{commensurators}.
\end{remark}

\begin{proposition}\label{Precompact}
    Let $(\iota_1,\iota_2)$ be a commensuration between profinite groups $G_1$ and $G_2$ with trivial virtual center over a profinite group $H$. Then $(\iota_1,\iota_2)$ admits a completion if and only if the subgroup $\langle \eta_1(G_1),\eta_2(G_2)\rangle$ of $\operatorname{Comm}(H)$ is precompact.
\end{proposition}

\begin{proof}
    Since $\operatorname{VZ}(G_i)=1$ for $i=1,2$, it follows directly that $\operatorname{VZ}(H)=1$. By Proposition \ref{commensurator}, the group $\operatorname{Comm}(H)$ is a l.c.t.d. group in which $H$ embeds as an open subgroup. If there is a completion $(j_1,j_2)$ for $(\iota_1,\iota_2)$, then Proposition \ref{commensurator} ensures the existence of continuous homomorphisms $\eta_1,\eta_2$ and $\eta_K$ such that $\eta_i$ is injective for $i=1,2$ and the diagram
\[
\begin{tikzcd}
H \arrow[r,"\iota_i"] \arrow[dr,"\eta_H"'] &
G_i \arrow[r,"j_i"] \arrow[d,"\eta_i"] &
K \arrow[dl,"\eta_K"] \\
&
\mathrm{Comm}(H) &
\end{tikzcd}
\] 
    commutes. In particular, $\eta_K(K)$ contains $G_i=\eta_i(G_i)$ for $i=1,2$ and thus $\langle\eta_1(G_1),\eta_2(G_2)\rangle$ is precompact. Conversely, if $\langle\eta_1(G_1),\eta_2(G_2)\rangle$ is precompact, then $(\eta_1,\eta_2)$ may be regarded as a completion for $(\iota_1,\iota_2)$.
\end{proof}

In particular:

\begin{corollary}
    Let $(\iota_1,\iota_2)$ be a commensuration between profinite groups $G_1$ and $G_2$ with trivial virtual center over a profinite group $H$. If $\operatorname{Comm}(H)$ is compact (profinite), then $(\iota_1,\iota_2)$ admits a completion.
\end{corollary}

The previous result applies to commensurations over the Nottingham group $\mathcal{N}(p)$ for $p>3$ \cite{ershov} and the so called \textit{hyperrigid} profinite groups, namely those for which the natural map $G\to\operatorname{Comm}(G)_S$ is an isomorphism \cite{commensurators}. In \cite[6.2]{commensurators}, the authors reformulate a result of Neukirch and Uchida and conclude that the absolute Galois group $G_\Q$ of the field of rational numbers is hyperrigid. For a number field $L$, we denote by $G_L$ its absolute Galois group.

\begin{corollary}
    A profinite group $G$ is commensurable with $G_\Q$ if and only if there is a finite normal subgroup $Q$ of $G$ such that $G/Q\cong G_L$ for some number field $L$. Also, if $\operatorname{VZ}(G)=1$, then $G\lesssim_o G_\Q$.
\end{corollary}

\begin{proof}
    If $G$ is commensurable with $G_\Q$, let $H$ be a common open subgroup of $G$ and $G_\Q$. The group $\operatorname{VZ}(G)$ is finite and $G/\operatorname{VZ}(G)$ is isomorphic to an open subgroup of $\operatorname{Comm}(H)\cong G_\Q$ by the universal property of the commensurator. The converse is clear.
\end{proof}

Observe that one implication of \ref{Precompact} holds in general:

\begin{proposition}
    If a commensuration $(\iota_1,\iota_2)$ between profinite groups $G_1$ and $G_2$ over a profinite group $H$ admits a completion, then the subgroup $\langle \eta_1(G_1),\eta_2(G_2)\rangle$ of $\operatorname{Comm}(H)$ is precompact.
\end{proposition}

\begin{proof}
    One needs to be careful since in general $G:=\operatorname{Comm}(H)$ is not Hausdorff. Let $(j_1,j_2)$ be a completion of $(\iota_1,\iota_2)$, with common codomain $K$. By similar arguments to those in Proposition \ref{Precompact}, $A:=\langle\eta_1(G_1),\eta_2(G_2)\rangle\leq\eta_K(K)$. Denote by $\pi:G\to G/\overline{\{1\}}$ the Hausdorff quotient map of $G$. Then $\overline{A}=\pi^{-1}(\overline{\pi(A)})$, and $\overline{\pi(A)}$ is compact since it is a closed subgroup of $\pi\circ\eta_K(K)$, which is compact and Hausdorff. This implies that $\overline{A}$ is compact. 
\end{proof}


\section{Commensurating graphs of pro-$\C$ groups}\label{sec:graphs}

A graph of groups is \textit{commensurating} if each edge group has finite-index image in its adjacent vertex groups. In the Appendix of \cite{Minasyan}, Minasyan considers the following generalization of the problem of completion of a commensuration. Let $(\mathcal{G},\Gamma)$ be a finite commensurating graph of groups, and let $\pi_1(\mathcal{G},\Gamma)$ be its fundamental group. When do there exist a group $Q$ and a homomorphism $\psi:\pi_1(\mathcal{G},\Gamma)\to Q$ whose restriction to each vertex group is injective and has finite-index image in $Q$?

When the graph $\Gamma$ consists of a single edge with two distinct vertices, the problem above reduces to the question of the existence of a completion for a commensuration. In this section, we establish pro-$\C$ analogs of Corollaries A.8 and A.9 of \cite{Minasyan}, extending Theorem \ref{TeoA} to finite commensurating graphs of groups.

\subsection{Finite graphs of pro-$\C$ groups.}

We now introduce several concepts and standard results concerning the theory of profinite groups acting on (profinite) graphs. The standard reference on this topic is \cite{RibG}, although we will restrict ourselves to a much less general setting than in the cited book.

A \textit{graph} $\Gamma$ is a set that is a disjoint union of subsets $V$ and $E$, together with maps
\[
    d_0,d_1:\Gamma\to V,
\]
which restrict to the identity on $V$. We call $V$ the vertex set of $\Gamma$ and $E$ the edge set of $\Gamma$. A \textit{finite graph of pro-$\C$ groups} $(\mathcal{G}, \Gamma)$ over a finite graph $\Gamma$ consists of pro-$\C$ groups $\mathcal{G}(m)$ for each $m\in\Gamma$ and continuous monomorphism $\partial_i:\mathcal{G}(e)\to\mathcal{G}(d_i(e))$ for every $e\in E$ and $i=1,2$. We call a finite graph of pro-$\C$ groups $(\mathcal{G},\Gamma)$ \textit{commensurating} if $\Gamma$ is connected and $\partial_i(\mathcal{G}(e))$ is open in $\mathcal{G}(d_i(e))$ for every $e\in E$ and $i=0,1$.

Let $(\mathcal{G},\Gamma)$ be a finite connected graph of pro-$\C$ groups. Fix a maximal subtree $T$ of $\Gamma$. The \textit{fundamental pro-$\C$ group $\Pi_1^{\C}(\mathcal{G},\Gamma)$} of $(\mathcal{G},\Gamma)$ is the pro-$\C$ group
\[
    \left(F_\C\amalg\coprod_{v\in V}\mathcal{G}(v)\right)/N,
\]
where $F_\C$ is the free pro-$\C$ group with basis $\{t_e:e\in E\}$ and $N$ is the closed normal subgroup generated by the set
\[
    \{t_e,e\in E\cap T\}\cup\{\partial_0(g)^{-1}t_e\partial_1(g)t_e^{-1},g\in\mathcal{G}(e),e\in E\}.
\]
Here, $\amalg$ denotes the free pro-$\C$ product. At first glance, the definition of $\Pi_1^\C(\mathcal{G},\Gamma)$ seems to depend on the choice of $T$, but it does not up to isomorphism.

Notice that by the definition of $\Pi_1^\C(\mathcal{G},\Gamma)$, we have a natural homomorphism
\[
    \varphi_v:\mathcal{G}(v)\to\Pi_1^\C(\mathcal{G},\Gamma)
\]
for each $v\in V$. We denote $\varphi_v(\mathcal{G}(v))$ by $\Pi(v)$ and we call the graph of pro-$\C$ groups $(\mathcal{G},\Gamma)$ \textit{injective} if $\varphi_v$ is a monomorphism for each $v\in V$ (in which case we identify $\Pi(v)$ with $\mathcal{G}(v)$). When $\Gamma$ consists of a single edge $e$ with distinct incidence vertex $u$ and $v$, the property of injectiveness of $(\mathcal{G},\Gamma)$ is the same as the properness of the free pro-$\C$ amalgamated product $\mathcal{G}(v)\amalg_{\mathcal{G}(e)}\mathcal{G}(u)$.



Following \cite{Minasyan}, we introduce the following definition:

\begin{definition}\label{def:tame}
    Let $(\mathcal{G},\Gamma)$ be a finite injective commensurating graph of pro-$\C$ groups with fundamental pro-$\C$ group $\Pi:=\Pi_1^\C(\mathcal{G},\Gamma)$ and let $N$ be a closed normal subgroup of $\Pi$. We say that $(\mathcal{G},\Gamma)$ is \textit{tame over $N$} if
    \begin{itemize}
        \item[(i)] $N$ is contained as an open subgroup of each vertex and edge group of $(\mathcal{G},\Gamma)$;
        \item[(ii)] the image of $\Pi$ in $\operatorname{Out}N$ induced by conjugation is finite. 
    \end{itemize}
    We say that $(\mathcal{G},\Gamma)$ is \textit{tame} if it is tame over
    some $N$.
\end{definition}


It is not clear at first that for an arbitrary finite injective commensurating graph of pro-$\C$ groups there exists a normal subgroup of its fundamental pro-$\C$ group satisfying $(i)$. The next lemma, which is a version of \cite[4.4]{BeZa}, ensures the existence of such an $N$. This may be viewed as a generalization of Lemma \ref{PropernessLem}.

\begin{lemma}\label{GraphCore}
    Let $(\mathcal{G},\Gamma)$ be a finite commensurating graph of pro-$\C$
    groups with fundamental pro-$\C$ group $\Pi$.
    Then $\Pi$ contains a closed normal subgroup $N$ such that
    $N$ is an open subgroup of $\Pi(e)$, for every $e\in E$.

\end{lemma}

\begin{proof}
    Let $T$ be a maximal subtree of $\Gamma$. If $e\in E\cap T$ and $u_i=d_i(e), i=0,1$, then the groups $\Pi(u_0)$ and $\Pi(u_1)$ are commensurable. Since $T$ is connected and contains all vertices of $\Gamma$, the groups $\Pi(u)$ are pairwise commensurable for all $u\in V$. Since $\Gamma$ is finite, the group
    \[  
        H:=\bigcap_{e\in E}\Pi(e)
    \]
    is a closed subgroup of $\Pi$ which is open in all edge and vertex groups. Notice that the subgroups $H^{t_e}$ are commensurable to $H$ for all stable letters $t_e$ ($e\in E-T$). 

    For each $v\in V$, let $U_v$ be an open normal subgroup of $\Pi$ such that $U_v\cap \Pi(v)\leq H$ and for each $e\in E-T$, let $W_e$ be an open normal subgroup of $\Pi$ with $W_e\cap H\leq H\cap H^{t_e}$. Let $U:=\bigcap_{v\in V} U_v \cap\bigcap_{e\in E-T} W_e$. We claim that $N:=U\cap H$ is the desired subgroup. It is clear that $N$ is an open subgroup of $\Pi(e)$ for every $e\in E$. It remains to show that $N$ is normal in $\Pi$.

    Since $H\leq\Pi(v)$ and $U\cap\Pi(v)\leq H$, we have that $N=U\cap\Pi(v)$ for every $v\in V$. Hence, if $g\in\Pi(v)$, $N^g=(U\cap\Pi(v))^g=U\cap \Pi(v)=N$ since $U$ is normal in $\Pi$. Also, if $t_e$ is a stable letter (that is, $e\in E-T$), then $N^{t_e}=(U\cap H)^{t_e}=U\cap H^{t_e}\geq (U\cap H)\cap (H\cap H^{t_e})=U\cap H=N$ since $U\cap H\leq W_e\cap H\leq H\cap H^{t_e}$. This implies that $N^{t_e}=N$ since a closed subgroup of a profinite group cannot be conjugate to a proper subgroup. Since $\Pi=\langle\Pi(v),t_e\rangle$, it follows that $N$ is normal in $\Pi$.
\end{proof}

\begin{remark}
    The previous lemma ensures the existence of a normal subgroup of $\Pi_1^\C(\mathcal{G},\Gamma)$ satisfying condition $(i)$ of Definition \ref{def:tame} for any finite injective commensurating graph of pro-$\C$ groups $(\mathcal{G},\Gamma)$, so that the main condition of tameness is condition $(ii)$.
\end{remark}

\begin{remark}\label{QuotVF}
    Let $(\mathcal{G},\Gamma)$ be a finite injective commensurating graph of pro-$C$ groups and $N$ a normal subgroup of $\Pi_1^\C(\mathcal{G},\Gamma)$ satisfying condition $(i)$ of Definition \ref{def:tame}. We may consider the graph of pro-$\C$ groups $(\overline{\mathcal{G}},\Gamma)$ with vertex and edge groups $\mathcal{G}(v)/N$ and $\mathcal{G}(e)/N$, respectively, and the natural monomorphisms between edge and vertex groups. The vertex groups of $(\overline{\mathcal{G}},\Gamma)$ are finite groups in $\C$, and $\Pi_1^\C(\overline{\mathcal{G}},\Gamma)\cong\Pi_1^\C(\mathcal{G},\Gamma)/N$ is virtually a free pro-$\C$ group of finite rank.
\end{remark}

\begin{theorem}\label{A.8}
    Let $(\mathcal{G},\Gamma)$ be an finite injective commensurating graph of pro-$\C$ groups with fundamental pro-$\C$ group $\Pi$. If $(\mathcal{G},\Gamma)$ is tame over $N$, then 
    \begin{itemize}
        \item[(i)] there is a closed normal subgroup $M$ of $\Pi$ which is free pro-$\C$ of finite rank such that $\mathcal{G}(v)\cap M =1$ and $\mathcal{G}(v) M$ is open in $\Pi$ for every $v\in V$.
        \item[(ii)] the pro-$\C$ group $\Pi$ embeds as an open subgroup of $\Pi/M\times \Pi/N$.
    \end{itemize}
\end{theorem}

\begin{proof}
    The proof is similar to as the proof of Theorem \ref{TeoA}. By Remark \ref{QuotVF}, the group $\Pi/N$ is virtually a free pro-$\C$ group of finite rank. Let $J$ be an open normal subgroup of $\Pi/N$ which is free pro-$\C$.

    Let $\pi:\Pi\twoheadrightarrow\Pi/N$ be the canonicial projection. Since $J$ is free, the short exact sequence
    \[
        1\to N\to\pi^{-1}(J)\to J\to 1
    \]
    splits as a semidirect product $N\rtimes J'$, where $J'$ is a lift of $J$ to $\pi^{-1}(J)$. By condition $(ii)$ of Definition \ref{def:tame}, the image of $\Pi$ in $\operatorname{Out}N$ is finite, so Lemma \ref{LemmaOut} yields a subgroup $K'$ of $NJ'$ such that $NK'=N\times K'$ and $K'$ is mapped to an open subgroup of $J$.

    Since $J$ is finitely generated, it has only finitely many open subgroups of index $[J:\pi(K')]$. We may therefore pick an open characteristic subgroup $\overline{K}$ of $J$ inside $\pi(K')$. We set $K:=\pi^{-1}(\overline{K})\cap K'$. Since $\overline{K}$ is characteristic in $J$, it is normal in $\Pi/N$ and therefore $\pi^{-1}(\overline{K})=NK=N\times K$ is an open normal subgroup of $\Pi$. Note that $pi|_{K'}$ is injective, so $K\cong\overline{K}$ is a free pro-$\C$ group of finite rank.

    We are in the setting of Theorem \ref{TechA}: $N$ is a normal subgroup of $\Pi$, and $K$ is a finitely generated subgroup of $\Pi$ such that $NK=N\times K$ is an open normal subgroup of $\Pi$. Therefore, we can choose a subgroup $M$ of $NK$ such that $M$ is normal in $\Pi$, $NM=N\times M$ is open, and $\pi(M)\cong M$ is open in $J$. We immediately conclude that $M$ is free pro-$\C$ of finite rank.

    Let $v\in V$ and $x\in\mathcal{G}(v)\cap M$. Then $\pi(x)$ lies both in the finite subgroup $\pi(\mathcal{G}(v))=\mathcal{G}(v)N/N\leq\Pi/N$ of $\Pi/N$ and in $\pi(M)\cong M$, which is torsion-free. Therefore $\pi(x)=1$ so that $x=1$ since $M\cap N=1$. This proves $(i)$.

    For $(ii)$, we can consider the natural monomorphism
    \[
        \varphi:\Pi\to\Pi/M\times\Pi/N.
    \]
    The group $G:=\varphi(\Pi)$ contains $\varphi(N)=NM/M\times 1$ and $\varphi(M)=1\times NM/N$. It follows that $MN/M\times MN/N\leq G$, which implies that $G$ is open in $\Pi/M\times\Pi/N$ (since $MN$ is open in $\Pi$).
\end{proof}

\begin{theorem}\label{A.9}
    Let $(\mathcal{G},\Gamma)$ be a finite commensurating graph of pro-$\C$ groups with fundamental pro-$\C$ group $\Pi$. The following are equivalent:
    \begin{itemize}
        \item[(i)] $(\mathcal{G},\Gamma)$ is tame (and, in particular, injective);
        \item[(ii)] there exists a pro-$\C$ group $Q$ and a continuous homomorphism $\psi:\Pi\to Q$ such that for every $v\in V$, $\psi|_{\mathcal{G}(v)}$ is injective and $\psi(\mathcal{G}(v))$ is open in $Q$.
    \end{itemize}
\end{theorem}

\begin{proof}
    If $(\mathcal{G},\Gamma)$ is tame over $N$, then Theorem \ref{A.8} gives a closed normal subgroup $M$ of $\Pi$ such that for all $v\in V$, $\mathcal{G}(v)\cap M=1$ and $\mathcal{G}(v)M$ is open in $\Pi$. Take $Q=\Pi/M$, and let $\psi$ be the canonical projection, it is clear that $\psi|_{\mathcal{G}(v)}$ is injective and $\psi(\mathcal{G}(v))$ is open in $Q$ for each $v\in V$.

    Conversely, if $(ii)$ holds, it is clear that $(\mathcal{G},\Gamma)$ is injective. By Lemma \ref{GraphCore}, there is a closed normal subgroup $N$ of $\Pi$ which is open in each vertex group of $(\mathcal{G},\Gamma)$. It remains to check that the subgroup $N$ satisfies condition $(ii)$ of Definition \ref{def:tame}.

    Let $P:=\psi(\Pi)$, so that $P$ contains $\psi(N)$ as an open normal subgroup. The homomorphism $P\to\operatorname{Out}\psi(N)$ induced by conjugation factors through the finite group $P/\psi(N)$, and therefore its image is finite. The result follows from the commutativity of the following diagram:

    \[
        \begin{tikzcd}
            \Pi \arrow[r] \arrow[d] &
            P \arrow[d] \\
            \operatorname{Out}(N) \arrow[r] &
            \operatorname{Out}(\psi(N))
        \end{tikzcd}
    \]
    and the fact that the map $\operatorname{Out}(N)\to\operatorname{Out}(\psi(N))$ is an isomorphism.
\end{proof}

\begin{remark}
    When $\Gamma$ consists of a single edge with two vertices, condition
    $(ii)$ of Theorem \ref{A.9} asserts precisely that the commensuration
    $(\partial_0,\partial_1)$ admits a completion, and injectivity of
    $(\mathcal{G},\Gamma)$ amounts to the properness of the free pro-$\C$ amalgamated product. Theorem \ref{A.9} thus recovers Theorem
    \ref{TeoA}.
\end{remark}

\noindent \textit{Acknowledgments.} The author thanks Pavel Zalesskii for suggesting that this paper be written and for fruitful discussions, and Ashot Minasyan for valuable comments and suggestions..

\end{document}